\documentclass{amsproc}
\usepackage{amssymb, graphicx, xspace, color, subcaption}

\newtheorem{theorem}{Theorem}[section]
\newtheorem{lemma}[theorem]{Lemma}
\newtheorem{cor}[theorem]{Corollary}

\theoremstyle{definition}
\newtheorem{defn}[theorem]{Definition}

\theoremstyle{remark}

\numberwithin{equation}{section}

\newcommand{\re}{\mbox{Re}\,\xspace}
\newcommand{\inner}[1]{\left \langle #1 \right \rangle\xspace}
\newcommand{\nf}{\eqref{eq: normal form}\xspace}
\newcommand{\po}{\eqref{eq: periodic orbit}\xspace}
\newcommand{\norm}[1]{\left \lVert #1 \right \rVert\xspace}
\begin{document}

\title[Control by time delayed feedback near a Hopf bifurcation point]{Control by time delayed feedback near\\ a Hopf bifurcation point}

\author{S. Verduyn Lunel}
\address{}
\curraddr{}
\email{}
\thanks{}

\author{B. de Wolff}
\address{}
\curraddr{}
\email{}
\thanks{}



\date{}

\begin{abstract} 
In this paper we study the stabilization of rotating waves using time delayed feedback control. It is our aim to put some recent results in a broader context by discussing two different methods to determine the stability of the target periodic orbit in the controlled system: 1) by directly studying the Floquet multipliers and 2) by use of the Hopf bifurcation theorem. We also propose an extension of the Pyragas control scheme for which the controlled system becomes a functional differential equation of neutral type. Using the observation that we are able to determine the direction of bifurcation by a relatively simple calculation of the root tendency, we find stability conditions for the periodic orbit as a solution of the neutral type equation. 
\end{abstract}

\maketitle

Stabilization of motion is a subject of interest in applications, where one often wishes the observed motion to be stable. Pyragas control \cite{pyragascontrol}, a form of time--delayed feedback control, provides a method to stabilize unstable periodic solutions of ordinary differential equations which has been sucessfully implemented in experimental set-ups \cite{experimentalverification1, experimentalverification2}. It can also be used to stabilize rotating waves in lasers \cite{lasers} and in coupled networks \cite{couplednetworkspyragas}. 
To be able to apply Pyragas control in physical applications, one is of course interested for which strength of the control term stability can be achieved.
Furthermore, in physical set-ups it is also relevant to have knowledge of the overall dynamics of the controlled system. Since by applying Pyragas control we turn a finite dimensional system into an infinite dimensional system, one expects the dynamics of the system to change singificantly. Therefore, the controlled system is an interesting object of study in itself \cite{oddnumberlimitation}. 

Various variations to the Pyragas control scheme have been proposed as well. For example, in \cite{extendedcontrol} the control term contains an infinite number of delay terms in which each delay is chosen to be a multiple of the period of the target periodic orbit; and in \cite{periodiccontrolgain} the control matrix is chosen to be non--autonomous. \\

In this article we continue an analysis started in \cite{oddnumberlimitation} and apply Pyragas control to the differential equation
\begin{align} \label{eq: uncontrolled normal form}
\dot{z}(t) = (\lambda + i) z(t) + (1 + i \gamma) \left| z(t) \right|^2 z(t)
\end{align}
where $\lambda, \gamma \in \mathbb{R}$ are parameters and $z: \mathbb{R} \to \mathbb{C}$. 
Solutions of the form $A(x,t) = z(t)e^{i \alpha x}$ of the Ginzburg--Landau equation
\begin{align*}
\frac{\partial A}{\partial t} (x, t) =(\lambda + i) \frac{\partial^2}{\partial x^2} A(x, t) + (1 + i \gamma) \left| A(x, t)\right|^2 A(x, t), \quad x \in \mathbb{R}, t \geq 0
\end{align*}
reduce, after rescaling, to solutions of \eqref{eq: uncontrolled normal form} \cite{ginzburglandau}. Equation \eqref{eq: uncontrolled normal form} can be used to model a range of physical phenomena, and arises as a model for Stuart-Landau oscillators \cite{clustersynchronization, stuartlandauoscillators} and laser dynamics \cite{lasers}. 

A useful property of \eqref{eq: uncontrolled normal form} is that we can explicitly find a periodic solution and that we can analytically determine its stability. Indeed, for $\lambda < 0$, system \eqref{eq: uncontrolled normal form} has a periodic solution given by
\begin{align} \label{eq: periodic orbit}
z(t) = \sqrt{- \lambda} e^{i(1- \gamma \lambda) t}
\end{align}
with period $T = 2 \pi / (1 - \gamma \lambda)$. For $\gamma \lambda < 1$, \eqref{eq: periodic orbit} is unstable as a solution of \eqref{eq: uncontrolled normal form} (see Section \ref{sec: uncontrolled system}). For the controlled system we write
\begin{align} \label{eq: normal form}
\dot{z}(t) = (\lambda + i) z(t) + (1 + i \gamma) \left| z(t) \right|^2 z(t) - Ke^{i \beta} \left[z(t) - z(t - \tau) \right]
\end{align}
with $K \in \mathbb{R}, \tau \geq 0$ and $\beta \in [0, \pi]$. The controlled system is designed such that for $\tau = T = 2 \pi / (1- \gamma \lambda)$, the function \eqref{eq: periodic orbit} is still a solution of \eqref{eq: normal form}. 

In \cite{counterexample}, the periodic solution \eqref{eq: periodic orbit} of \eqref{eq: uncontrolled normal form} was used as a counterexample to the claim that periodic orbits with an odd number of Floquet multipliers outside the unit circle cannot be stabilized using Pyragas control. In \cite{oddnumberlimitation}, the bifurcation diagram of the controlled system \eqref{eq: normal form} was studied in more detail, and it was shown that the stability of \eqref{eq: periodic orbit} as a solution of \eqref{eq: normal form} can be determined using the Hopf bifurcation theorem. In fact, it was shown that the periodic solution \eqref{eq: periodic orbit} of the system \eqref{eq: normal form} emmanates from a Hopf bifurcation. By using the direction of the Hopf bifurcation (i.e. whether the Hopf bifurcation is sub-- or supercritical), one is then able, for $\lambda$ near the bifurcation point and given $\gamma$, to find conditions on the parameters $K, \beta$ that ensure that the periodic orbit \eqref{eq: periodic orbit} is stable as a solution of \eqref{eq: normal form}. \\

In Sections \ref{sec: uncontrolled system}--\ref{sec: lambda varieren}, we place the results from \cite{oddnumberlimitation} in a broader context using the theory developed for delay equations in \cite{delayequations} and, in particular discuss and compare different methods to determine the stability of \eqref{eq: periodic orbit} as a solution of \eqref{eq: normal form}. We start by exploring the dynamics of the uncontrolled system \eqref{eq: uncontrolled normal form} in Section \ref{sec: uncontrolled system}. In Section \ref{sec: multipliers} we give necessary conditions for \po to be stable as a solution of \nf by direct investigation of the Floquet multipliers. As a different approach to determine the stability of \eqref{eq: periodic orbit} as a solution of \eqref{eq: normal form}, we use -- inspired by \cite{oddnumberlimitation} -- the Hopf bifurcation theorem. In Section \ref{sec: pyragas curve} we approach the bifurcation point over a different curve in the parameter plane than was done in \cite{oddnumberlimitation}. This enables us to give stability conditions for a wider range of parameter values. We choose the curve through parameter plane in such a way that we a priori know for which points on the curve a periodic solution exists. A relatively simple calculation of the root tendency of the roots of the characterstic equation then directly yields the direction of the bifurcation. In Section \ref{sec: lambda varieren}, we give a direct proof of the result from \cite{oddnumberlimitation} using the explicit closed--form formula's to determine the direction of the Hopf bifurcation developed in \cite{delayequations}.

In Section \ref{sec: afgeleide control} we propose a variation to the Pyragas control scheme for which the controlled system becomes a functional differential equation of neutral type. We apply the proposed control scheme to the system \eqref{eq: uncontrolled normal form} and use the methods developed in Section \ref{sec: pyragas curve} to determine the stability of the target periodic orbit. 

\section{Dynamics of the uncontrolled system} \label{sec: uncontrolled system}
Before studying the dynamics of the uncontrolled systems, we make some remarks on terminology used throughout the article. 

\begin{defn}
Let $r > 0$, $\mathcal{S} = C([-r, 0], \mathbb{R}^n)$ equipped with the norm $\norm{\phi}_\infty = \sup_{\theta \in [-r, 0]} \left| \phi(\theta) \right|$. Let $F: \mathcal{S} \to \mathbb{R}^n$. Let us study the retarded functional differential equation 
\begin{align} \label{eq: delay equation stabiliteit}
\dot{x}(t) = F(x_t) \quad  t \geq 0
\end{align}
where $x_t(\theta) = x(t + \theta)$ for $\theta \in [-r, 0]$. Denote by $T(t)$ the semi--flow associated to \eqref{eq: delay equation stabiliteit}. Let $x_0$ be an equilibrium of \eqref{eq: delay equation stabiliteit}. Then we say that $x_0$ is \emph{stable} if it is asymptotically stable, i.e. the following two conditions are satisfied: 1) For every $\epsilon > 0$ there exists a $\delta > 0$ such that if $\norm{\phi - x_0}_\infty < \delta$ for $\phi \in \mathcal{S}$, then $\norm{T(t)\phi - x_0}_\infty < \epsilon$ for all $t \geq 0$. 2) There exists a $b > 0$ such that if $\norm{\phi - x_0} < b$ for $\phi \in \mathcal{S}$, then $\lim_{t \to \infty} \norm{T(t) \phi - x_0}_\infty =  0$. We say that $x_0$ is \emph{unstable} if it is asymptotically unstable. 
\end{defn}  

Note that we do not require exponential stability. However, when we determine that a fixed point is stable by establishing that all the associated eigenvalues are in the left half of the complex plane, exponential stability automatically follows.

To study the uncontrolled system \eqref{eq: uncontrolled normal form}, we can take the real and imaginary parts and view \eqref{eq: uncontrolled normal form} as a system on $\mathbb{R}^2$ given by
\begin{align} \label{eq: uncontrolled normal form on r2}
\begin{pmatrix}
\dot{x}(t) \\ \dot{y}(t)
\end{pmatrix} 
 = \begin{pmatrix}
 \lambda & -1 \\
 1 & \lambda
 \end{pmatrix}
 \begin{pmatrix}
 x(t) \\ y(t)
 \end{pmatrix}
+ (x^2(t) + y^2(t)) \begin{pmatrix}
1 & - \gamma \\
\gamma & 1
\end{pmatrix}
\begin{pmatrix}
x(t) \\ y(t)
\end{pmatrix}
\end{align}
Note that $(x, y) = (0, 0)$ is an equilibrium of this system, and the linearization of \eqref{eq: uncontrolled normal form on r2} can be used to determine its stability. 

\begin{lemma} \label{lem: stability equilibrium uncontrolled system}
If $\lambda < 0$, the equilibrium $(x, y) = (0, 0)$ of \eqref{eq: uncontrolled normal form on r2} is stable. If $\lambda > 0$, the equilibrium $(x, y) = (0, 0)$ of \eqref{eq: uncontrolled normal form on r2} is unstable.
\end{lemma}
\begin{proof}
Linearzing the system \eqref{eq: uncontrolled normal form on r2} around the zero solution gives: 
\begin{align} \label{eq: linearized uncontrolled normal form}
\begin{pmatrix}
\dot{x}(t) \\ \dot{y}(t)
\end{pmatrix}
 = \begin{pmatrix}
 \lambda & -1 \\
 1 & \lambda
 \end{pmatrix}
 \begin{pmatrix}
 x(t) \\ y(t)
 \end{pmatrix}
\end{align}
The eigenvalues of the matrix in the RHS of \eqref{eq: linearized uncontrolled normal form} are given by $\mu_{\pm} = \lambda \pm i$. This shows that the equilibrium point $(x, y) = (0, 0)$ is stable for $\lambda < 0$ and unstable for $\lambda > 0$.
\end{proof}

We recall that a \emph{Hopf bifurcation} of an equilibrium occurs if we have exactly one pair of non--zero roots at the imaginary axis, and that this pair of roots crosses the axis with non--zero speed as we vary the parameters. Indeed, in the case of \eqref{eq: uncontrolled normal form on r2} we see that for $\lambda = 0$, the eigenvalues $\mu_{\pm}$ cross the imaginary axis at non--zero speed, since $\frac{d}{d \lambda} \re( \mu_{\pm} (\lambda)) = 1 \neq 0$. Thus, we find that for $\lambda = 0$ a Hopf bifurcation of the origin of system \eqref{eq: uncontrolled normal form} takes place. The Hopf bifurcation theorem now implies that for parameter values $\lambda$ near the bifurcation point $\lambda = 0$, an unique periodic solution of \eqref{eq: uncontrolled normal form on r2} exists. 

It turns out that we can explicitly compute this periodic solution of \eqref{eq: uncontrolled normal form on r2}. By substituting $z(t) = r(t) e^{i \phi(t)}$ into \eqref{eq: uncontrolled normal form} with $r(t), \phi(t) \in \mathbb{R}$, we find that for $\lambda < 0$ a periodic solution of \eqref{eq: uncontrolled normal form} is given by \eqref{eq: periodic orbit}. Using that we know for which parameter values $\lambda$ a periodic orbit exists, we can easily determine whether the Hopf bifurcation is sub-- or supercritical. This is summarized for retarded functional differential equations in the following theorem. 

\begin{theorem} \label{lem: direction hopf bifurcation by roots of characteristic equation}
Let us study the system
\begin{align} \label{eq: rfde}
\dot{x}(t) = F(\lambda, x_t)
\end{align}
where $r > 0, \lambda \in \mathbb{R}$, $F: \mathcal{C}([-r, 0], \mathbb{R}^n) \times \mathbb{R} \to \mathbb{R}^n$ satisfies $F(0, \lambda) = 0$ for all $\lambda \in \mathbb{R}$ and $x_t$ is defined as $x_t(\theta) = x(t+ \theta)$ for $\theta \in [-r, 0]$. Let us assume that for $\lambda = \lambda_0$ a Hopf bifurcation of the origin of system \eqref{eq: rfde} takes place. Let us write $\Delta(\mu, \lambda)$ for the characteristic equation of the linearization of \eqref{eq: rfde}. Denote by $\mu_0 = \mu_0(\lambda)$ the root of the characteristic equation $\Delta(\mu_0 (\lambda), \lambda) = 0$ that satisfies $\mu_0(\lambda_0) = i \omega_0$ for some $\omega_0 \in \mathbb{R}$. Furthermore, let us assume that for $\lambda < \lambda_0$, a periodic solution $\overline{x}_\lambda$ of the system \eqref{eq: rfde} exists. Then we find that the Hopf bifurcation is subcritical if $\re(\mu_0 (\lambda)) < 0$ for $\lambda < \lambda_0$ in a neighbourhood of $\lambda_0$; the Hopf bifurcation is supercritical if $\re(\mu_0(\lambda)) > 0$ for $\lambda < \lambda_0$ in a neighbourhood of $\lambda_0$. 
\end{theorem}

\begin{proof}
Since by assumption for $\lambda = \lambda_0$ a Hopf bifurcation of the origin of system \eqref{eq: rfde} takes place, we find by the Hopf bifurcation theorem (see for example \cite{introductionfde} for the Hopf bifurcation theorem for retarded functional differential equations) that an unique periodic solution of \eqref{eq: rfde} exists for parameters $\lambda$ near the bifurcation point $\lambda = \lambda_0$. Since $\overline{x}_\lambda$ is a periodic solution of \eqref{eq: rfde} for $\lambda < \lambda_0$, we conclude that this periodic solution arises from the Hopf bifurcation at $\lambda = \lambda_0$. 

If now $\re(\mu_0 (\lambda)) < 0$ for $\lambda < \lambda_0$ in a neighbourhood of $\lambda_0$, we find that the periodic solution arising from the Hopf bifurcation exists for parameter values $\lambda$ for which $\mu_0(\lambda)$ is in the left half of the complex plane. This implies that the Hopf bifurcation is subcritical. Similarly, if $\re(\mu_0(\lambda)) > 0$ for $\lambda < \lambda_0$ in a neighbourhood of $\lambda_0$, we find that the periodic solution arising from the Hopf bifurcation exists for parameters $\lambda$ for which $\mu_0(\lambda)$ is in the right half of the complex plane. This implies that the Hopf bifurcation is supercritical. 
\end{proof}

Since in the case of system \eqref{eq: uncontrolled normal form} a periodic solution exists for $\lambda < 0$, combining Lemma \ref{lem: stability equilibrium uncontrolled system} with Lemma \ref{lem: direction hopf bifurcation by roots of characteristic equation} yields the following corollary: 

\begin{cor} \label{cor: stability periodic solution of uncontrolled system by Hopf}
The Hopf bifurcation at $\lambda =0$ of system \eqref{eq: uncontrolled normal form} is subcritical and the periodic solution \eqref{eq: periodic orbit} of \eqref{eq: uncontrolled normal form} is unstable for parameters $\lambda < 0$ near the bifurcation point $\lambda = 0$. 
\end{cor}

We see that the Hopf bifurcation theorem gives us information on the stability of the periodic solution \eqref{eq: periodic orbit} of \eqref{eq: uncontrolled normal form} for parameters in $\lambda < 0$ in a neighbourhood of the bifurcation point $\lambda =0$. 

For general parameters $\lambda < 0$, the stability of the periodic orbit \po of \eqref{eq: uncontrolled normal form} is determined by its Floquet multipliers.

\begin{lemma} \label{lem: stability periodic solution of uncontrolled system}
Let $\lambda < 0$. Then the periodic solution \eqref{eq: periodic orbit} of \eqref{eq: uncontrolled normal form} is stable if $\gamma \lambda > 1$ and unstable if $\gamma \lambda < 1$. 
\end{lemma}
\begin{proof}
In order to compute the Floquet multipliers, we first compute the linear variational equation. As it turns out that the linear variational equation is autonomous, the computation of the Floquet multipliers is then relatively straightforward. 

As in \cite{oddnumberlimitation}, we write small deviations around the periodic solution \eqref{eq: periodic orbit} as 
\begin{align} \label{eq: small deviation linear variational equation}
z(t) = R_p e^{i \omega_p t}[1 + r(t) + i \phi(t)]
\end{align}
with $r(t), \phi(t) \in \mathbb{R}$ and where $R_p = \sqrt{-\lambda}$ denote the radius and $ \tau_p = 1 - \gamma \lambda$ the angular frequence of \eqref{eq: periodic orbit}. For \eqref{eq: small deviation linear variational equation} to be a solution of \eqref{eq: uncontrolled normal form}, we should have that
\begin{equation}
\begin{aligned} \label{eq: linear variational equation all orders}
i \omega_p R_p e^{i \omega_p t} &\left(1 + r(t) + i \phi(t))\right) + R_p e^{i \omega_p t} \left(\dot{r}(t) + i \dot{\phi}(t) \right) \\&= (\lambda + i) R_p e^{i \omega_p t} \left(1 + r(t) + i \phi(t) \right) \\ &+ (1 + i \gamma) R_p^3 e^{i \omega_p t} \left| 1 + r(t) + i \phi(t) \right|^2 (1 + r(t) + i \phi(t) ) 
\end{aligned}
\end{equation}
Up to first order, this expression reduces to
\begin{equation}
\begin{aligned}
 \label{eq: linear variational equation not simplified}
i \omega_p R_p e^{i \omega_p t}& (1 + r(t) + i \phi(t)) + R_p e^{i \omega_p t} \left(\dot{r}(t) + i \dot{\phi}(t) \right) \\
&= (\lambda + i ) R_p e^{i \omega_p t} (1 + r(t) + i \phi(t) ) + (1 + i \gamma) R_p^3 e^{i \omega_p t} (1 + 3 r(t) + i \phi(t))
\end{aligned}
\end{equation}
Using that \eqref{eq: periodic orbit} is a solution of \eqref{eq: uncontrolled normal form}, we arrive at
\begin{align*}
i \omega_p R_p e^{i \omega_p t} = (\lambda + i) R_p e^{i \omega_p t} + (1 + i \gamma) R_p^3 e^{i \omega_p t}
\end{align*}
Cancelling out factors $R_p e^{i \omega_p t}$ on both sides of \eqref{eq: linear variational equation not simplified}, we have
\begin{align*}
i \omega_p (r(t) + i \phi(t)) + \dot{r}(t) + i \dot{\phi}(t) = (\lambda + i) (r(t) + i \phi(t)) + (1 + i \gamma) R_p^2 (3 r(t) + i \phi(t))
\end{align*}
Using that $R_p^2  = - \lambda$ and  $\omega_p = 1 - \gamma \lambda$, leads to the linear variational equation
\begin{align} \label{eq: linear variational equation ODE complex valued}
\dot{r}(t) + i \dot{\phi} = - 2 \lambda r(t) - 2 i \gamma \lambda r(t)
\end{align}
Taking real and imaginary parts, the linear system on $\mathbb{R}^2$ is given by
\begin{align} \label{eq: linear variational equation ODE}
\begin{pmatrix}
\dot{r}(t) \\
\dot{\phi}(t) 
\end{pmatrix}
& = \begin{pmatrix}
- 2 \lambda & 0 \\
- 2 \gamma \lambda & 0 
\end{pmatrix}
\begin{pmatrix}
r(t) \\ \phi(t)
\end{pmatrix}
\end{align}
Put
\begin{align*}
A = \begin{pmatrix}
- 2 \lambda &  0 \\
- 2 \gamma \lambda & 0
\end{pmatrix}.
\end{align*}
The Floquet multipliers of {eq: linear variational equation ODE} are given by 
\begin{align*}
\lambda_i = e^{\lambda_i T} \quad i = 1, 2
\end{align*}
where $\lambda_1, \lambda_2$ are the eigenvalues of $A$ and $T = \frac{2 \pi}{1 - \gamma \lambda}$ the minimal period of the periodic solution \eqref{eq: periodic orbit}. The eigenvalues of $A$ are given by $\lambda_1 = 0, \lambda_2 = -2 \lambda$; therefore $\mu_1 = 1$ (the trivial Floquet multiplier) and
\begin{align*}
\mu_2 = e^{-2 \lambda \frac{2 \pi}{1- \gamma \lambda}}
\end{align*}
Since the periodic orbit exists for $\lambda < 0$, we conclude that the periodic orbit \eqref{eq: periodic orbit} of \eqref{eq: uncontrolled normal form} is stable if $\gamma \lambda > 1$ and unstable if $\gamma \lambda < 1$. 
\end{proof}

We now note that the results of Lemma \ref{lem: stability periodic solution of uncontrolled system} are consistent with Corollary \ref{cor: stability periodic solution of uncontrolled system by Hopf}. 
If $\gamma \geq 0$, Lemma \ref{lem: stability periodic solution of uncontrolled system} implies that the periodic solution \eqref{eq: periodic orbit} of \eqref{eq: uncontrolled normal form} is unstable for all $\lambda < 0$. If $\gamma < 0$, we find that \eqref{eq: periodic orbit} is unstable as a solution of \eqref{eq: uncontrolled normal form} for $\frac{1}{\gamma} < \lambda < 0$ and stable for $\lambda < \frac{1}{\gamma}$. In particular, we always find that \eqref{eq: periodic orbit} is unstable as a solution of \eqref{eq: uncontrolled normal form} for $\lambda < 0$ in a neighbourhood of $\lambda = 0$, as asserted by Corollary \ref{cor: stability periodic solution of uncontrolled system by Hopf}. 

\section{Floquet multipliers in the controlled system} \label{sec: multipliers}
In Section \ref{sec: uncontrolled system}, we used Floquet theory to determine the stability of the periodic solution \po as a solution of the ODE \eqref{eq: uncontrolled normal form}. As we have seen in Lemma \ref{lem: stability periodic solution of uncontrolled system}, the linear variational equation becomes autonomous in this case, and the computation of the Floquet multipliers reduces to the calculation of eigenvalues of a $2 \times 2$--matrix. 

In this section we use Floquet theory to gain information on the stability of \po as a solution of the delay equation \nf. We again find that the linear variational equation is autonomous, but the computation of the Floquet multipliers is more involved, because the characteristic matrix function now becomes transcendental. We will first present a necessary condition for \po to be stable as a solution of \nf, and then, in Sections \ref{sec: pyragas curve} and \ref{sec: lambda varieren}, we use the Hopf bifurcation theorem to show that for $\lambda < 0$ small, this condition is also sufficient. 

\begin{lemma}
Let us consider the system \nf with $\gamma \lambda < 0$. 
A necessary condition for \po to be stable as a solution of \nf, is that
\begin{align*}
1 + \tau K (\cos \beta + \gamma \sin \beta) < 0
\end{align*}
with $\tau = \frac{2 \pi}{1- \gamma \lambda}$ the minimal period of \po. 
\end{lemma}
\begin{proof}
We start by determining the linear variational equation of \nf around the periodic solution \po by writing small deviations around the solution \po as in \eqref{eq: small deviation linear variational equation}. 

We note that we go from system \eqref{eq: uncontrolled normal form} to system \nf by adding the linear term $Ke^{i \beta} \left[z(t) - z(t - \tau) \right]$. 
Using that we already determined the linearization of system \eqref{eq: linear variational equation all orders} around the periodic solution \po in the proof of Lemma \ref{lem: stability periodic solution of uncontrolled system}, we find that the linearization of system \nf around the solution \po satisfies
\begin{align*}
\dot{r}(t) + i \dot{\phi}(t) = -2 \lambda r(t) - 2 i \gamma \lambda \phi(t) - Ke^{i \beta} \left[r(t) + i \phi(t) - r(t - \tau) - i \phi(t - \tau) \right]
\end{align*}
where $\tau = \frac{2 \pi}{1 - \gamma \lambda}$ is the period of the solution \po. 
Taking real and imaginary parts, we see that the linear variational equation of system \nf around the solution \po is given by
\begin{align}
\begin{pmatrix}
\dot{r}(t) \\
\dot{\phi}(t)
\end{pmatrix} = \begin{pmatrix}
-2 \lambda & 0 \\
-2 \lambda \gamma & 0
\end{pmatrix} 
\begin{pmatrix}
r(t) \\ \phi(t)
\end{pmatrix} - K \begin{pmatrix}
\cos \beta & - \sin \beta \\
\sin \beta & \cos \beta 
\end{pmatrix} \begin{pmatrix}
r(t) - r(t - \tau) \\
\phi(t) - \phi(t - \tau)
\end{pmatrix} \label{eq: linear variational equation dde}
\end{align}
Note that the linear variational equation is autonomous. Therefore, the Floquet exponents are given by the roots of the characteristic equation corresponding to \eqref{eq: linear variational equation dde}. The characteristic function reads
\begin{equation}
\begin{aligned} \label{eq: ce controlled}
\det \Delta(\mu) & = ( \mu + 2 \lambda + K \cos \beta (1 - e^{- \mu \tau}) \left(\mu + K \cos \beta (1 - e^{- \mu \tau}) \right)  \\
&\qquad+ \left(2 \lambda \gamma + K \sin \beta (1 - e^{- \mu \tau}) \right) K \sin \beta (1 - e^{- \mu \tau})
\end{aligned}
\end{equation}
Observe that we have indeed a trivial Floquet multiplier, as predicted by Floquet theory, since $\det \Delta(0) = 0$ for all values of $\lambda, \gamma, K, \beta$. 

Let us now consider the stability of \po as a solution of \nf in the parameter plane $H = \{(\lambda, K) \mid \lambda < 0, K \in \mathbb{R} \}$ and fix a point $(\lambda_0, K_0) \in H$. For $K = 0$; system \nf reduces to \eqref{eq: uncontrolled normal form} and Lemma \ref{lem: stability periodic solution of uncontrolled system} gives that for $(\lambda, K) = (\lambda_0, 0)$  we have exactly one Floquet exponent in the right half of the complex plane. 

If a Floquet exponent moves from the right to the left half of the complex plane or vice versa, it should cross the imaginary axis \cite{introductionfde} If the Floquet exponent crosses the imaginary axis at the point $i \omega$ with $\omega \neq 0$, then the number of Floquet exponents in the right half of the complex plane changes by two, since if $\Delta(i \omega) = 0$, then also $\Delta(- i \omega) = 0$.

Now let us move from $(\lambda_0, 0)$ to the point $(\lambda_0, K_0)$ and suppose that we do not cross a point $(\lambda_0, K')$ such that for $\lambda = \lambda_0, K = K'$, $\mu = 0$ is a non--trivial solution of \eqref{eq: ce controlled}, then the previous remarks imply that on the way from $(\lambda_0, 0)$ to $(\lambda_0, K_0)$ the number of Floquet exponents can only change by an even number; since for $(\lambda_0, 0)$ the number of Floquet exponent is one, this gives that for $(\lambda_0, K_0)$ the number of Floquet exponents in the right half of the complex plane is odd. Since the number of Floquet multipliers in the right half of the complex plane is always non--negative, we see that it is at least one. Therefore, the periodic solution \po of \nf is unstable voor $(\lambda, K) =(\lambda_0, K_0)$. Thus, we find that a necessary condition for \po to be stable as a solution of \nf for $(\lambda, K) = (\lambda_0, K_0)$ is that on the way from $(\lambda_0, 0)$ to $(\lambda_0, K_0)$ we cross a point such that $\mu = 0$ is a non--trivial solution of \eqref{eq: ce controlled}. 

It holds that $\mu = 0$ is a non--trivial root of $\det \Delta(\mu) = 0$ if and only if $(\det \Delta(\mu))/\mu= 0$. Using \eqref{eq: ce controlled} gives that
\begin{align*}
\frac{\det \Delta(\mu)}{\mu} = \mu + 2 K \cos \beta (1- e^{- \mu \tau}) + 2 \lambda + 2 \lambda K \cos \beta \frac{1 - e^{- \mu \tau}}{\mu}\\ + K^2 \frac{(1 - e^{- \mu \tau})^2}{\mu} + 2 \lambda \gamma K \sin \beta \frac{1 - e^{- \mu \tau}}{\mu} 
\end{align*}
Combining this with
\begin{align*}
\frac{1 - e^{- \mu \tau}}{\mu} = \tau_p + \mathcal{O}(\mu)
\end{align*}
gives that $\mu = 0$ is a non--trivial root of $\det \Delta(\mu) = 0$ if and only if 
$$2 \lambda (1 + \tau K (\cos \beta + \gamma \sin \beta)) = 0.$$ 
For $\lambda < 0$, we now find that $\mu = 0$ is a non--trivial root of $\det \Delta(\mu) = 0$ if and only if $1 + 2 \tau K (\cos \beta + \gamma \sin \beta) = 0$. 

We note that the equation $1 + 2 \tau K (\cos \beta + \gamma \sin \beta) = 0$ defines a curve  $\ell$ in the parameter plane $H$. Let $(\lambda_0, K_0)$ be as above; since for $K = 0$ we have that $1 + 2 \pi K (\cos \beta + \gamma \sin \beta)  = 1 > 0$, we cross the curve $\ell$ on the way from $(\lambda_0, 0)$ to $(\lambda_0, K_0)$ if and only if $1 + \tau K( \cos \beta + \gamma \sin \beta) < 0$ for $(\lambda, K) = (\lambda_0, K_0)$. This proves the lemma. 
\end{proof}

\section{Hopf bifurcation and stability conditions} \label{sec: pyragas curve}
In the previous section, we used Floquet theory to determine necessary conditions for the periodic orbit \po of \nf to be stable. In this section, we use -- inspired by \cite{counterexample} and \cite{oddnumberlimitation} -- the Hopf bifurcation theorem to find sufficient conditions for the periodic orbit \po to be stable as a solution of \nf for parameter values near the bifurcation point. In particular, we find conditions for which the periodic solution \eqref{eq: periodic orbit} of \eqref{eq: normal form} arises from a Hopf bifurcation. Using that a Hopf bifurcation is either subcritical (an unstable periodic orbit arises for parameter values where the fixed point is stable) or supercritical (a stable periodic orbit arises for parameter values where the fixed point is unstable), we then determine for which parameter values \eqref{eq: periodic orbit} is (un)stable as a solution of \eqref{eq: normal form}. 

We note that in the Hopf bifurcation theorem (see Theorem \ref{thm: occurence of the Hopf bifurcation} below), the parameters are varied along a curve in parameter space. In order to apply the Hopf bifurcation theorem to system \eqref{eq: normal form}, we should therefore choose a one-dimensional curve through the parameter space to approach the bifurcation point. There are, of course, different ways to do this and different curves of approach will give us different information on the behaviour of the controlled system. In this section, the choice of curve is motivated by the fact that we know a priori for which parameter values in the $(\lambda, \tau)$--plane a periodic solution exists. 

Following \cite{oddnumberlimitation}, we introduce the following definitions: 

\begin{defn} \label{def: pyragas curve}
We define the \emph{Pyragas curve} as the curve in $(\lambda, \tau)$-parameter space given by the graph of $\tau(\lambda) = \frac{2 \pi}{1- \gamma \lambda}$ with $\lambda$ in the domain $(-\infty, 0) \backslash \{\frac{1}{\gamma} \}$. 
\end{defn} 
We note that for parameter values on the Pyragas curve, \eqref{eq: periodic orbit} is a solution of \eqref{eq: normal form}. 
\begin{defn} \label{def: extended pyragas curve}
We define the \emph{extended Pyragas curve} as the curve in $(\lambda, \tau)$-parameter space given by the graph of $\tau(\lambda) = \frac{2 \pi}{1- \gamma \lambda}$ with $\lambda$ in the domain $(-\infty, \frac{1}{\gamma})$ if $\gamma > 0$ and $\lambda$ in the domain $(- \frac{1}{\gamma}, \infty)$ if $\gamma < 0$. 
\end{defn}

In this section, we approach the point $(\lambda, \tau) = (0, 2\pi)$ over the extended Pyragas curve. We show that, under certain conditions on parameter values, we find a Hopf bifurcation of the origin for $(\lambda, \tau) = (0, 2 \pi)$. Uniqueness of the periodic orbit arising from the Hopf bifurcation now directly guarantees that the periodic orbit \eqref{eq: periodic orbit} of \eqref{eq: normal form} arises from a Hopf bifurcation for parameter values near the bifurcation point. 

\begin{figure}
\centering
\includegraphics[width = 0.6 \textwidth]{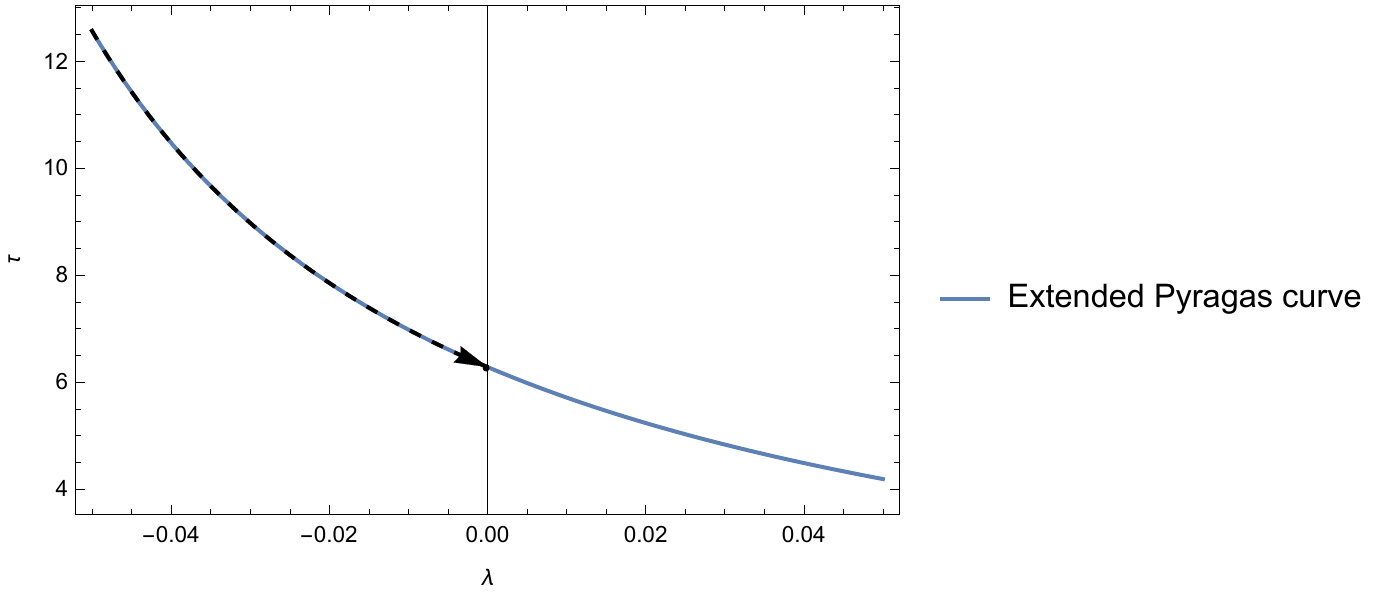}
\caption{The Hopf bifurcation point $(\lambda, \tau) =(0, 2 \pi)$ approached over the extended Pyragas curve.}
\end{figure}

We first state Theorem X.2.7 and Theorem X.3.9 from \cite{delayequations} on the Hopf bifurcation for differential delay equations. 

\begin{theorem}[Occurence of a Hopf bifurcation] \label{thm: occurence of the Hopf bifurcation}
Let us consider the differential delay equation 
\begin{align}
\begin{cases}
\dot{x}(t) = A(\mu) x(t) + B(\mu) x(t - \tau) + g(x_t, \mu) \quad & \mbox{for} \ t \geq 0 \\
x(t) = \phi(t) \quad & \mbox{for} \ - \tau \leq t \leq 0
\end{cases} \label{eq: dde stelling richting hopf bifurcatie}
\end{align}
where $\mu$ is a scalar parameter, $x_t \in \mathcal{C}([-\tau, 0], \mathbb{R}^n)$ is defined as $x_t(\theta) = x(t + \theta)$, $A(\mu), \ B(\mu)$ are $n \times n$-matrices, $\mu \mapsto A(\mu), \ \mu \mapsto B(\mu)$ are smooth maps, $g: \mathcal{C}([-\tau, 0], \mathbb{R}^n) \times \mathbb{R} \to \mathbb{R}^n$ is at least $\mathcal{C}^2$, $g(0, \mu) = D_1 g(0, \mu)$ for all $\mu$ and $\phi \in \mathcal{C}([-\tau, 0], \mathbb{R}^n)$. Denote the characteristic function of \eqref{eq: dde stelling richting hopf bifurcatie} by $\Delta(\lambda, \mu)$. Assume that there exists an $\omega_0 \in \mathbb{R}\backslash \{0 \}$ and a $\mu_0 \in \mathbb{R}$ such that $\Delta(i \omega_0, \mu_0) = 0$.  Let $p, q \in \mathbb{C}^n$ satisfy
\begin{align} \label{eq: definitie p, q, phi}
\Delta(i \omega_0, \mu_0) p = 0, \ \Delta(i \omega_0, \mu_0)^T q = 0, \ q D_1 \Delta(i \omega_0, \mu_0) p = 1
\end{align} 
If $\re(q \cdot D_2 \Delta(i \omega_0, \mu_0) p) < 0$, $i \omega_0$ is a simple root of $\Delta(z, \mu_0)$ and no other roots of $\Delta(z, \mu_0)$ belong to $i \omega_0 \mathbb{Z}$, a Hopf bifurcation of the origin of \eqref{eq: dde stelling richting hopf bifurcatie} occurs. 
\end{theorem}

We remark that the condition that $\re(q \cdot D_2 \Delta(i \omega_0, \mu_0) p) < 0$ ensures that the eigenvalue on the imaginary axis that exists for $\mu = \mu_0$, moves to the right half of the complex plane if we vary $\mu$. 

\begin{theorem}[Direction of the Hopf bifurcation] \label{thm: direction of the Hopf bifurcation}
Let us study the system \eqref{eq: dde stelling richting hopf bifurcatie} with $A, B, g, p, q, \mu_0$ and $\omega_0$ as in Theorem \ref{thm: occurence of the Hopf bifurcation}. 
If we introduce
\begin{align}
\mu_2 = \frac{\re(c)}{\re(q \cdot D_2 \Delta(i \omega_0, \mu_0) p) } \label{eq: derde orde hopf coefficient}
\end{align}
with
\begin{equation}
\begin{aligned}
c &= \frac{1}{2} q \cdot D_1^3 g(0, \mu_0)(\phi, \phi, \overline{\phi}) + q\cdot D_1^2 g(0, \mu_0) (e^{0 .} \Delta(0, \mu_0)^{-1} D_1^2 g(0, \mu_0) (\phi, \overline{\phi}), \phi)  \\
& + \frac{1}{2} q \cdot D_1^2 g(0, \mu_0)(e^{2 i \omega_0 .} \Delta(2 i \omega_0, \mu_0)^{-1} D_1^2 g(0, \mu_0) (\phi, \phi), \overline{\phi}) 
\end{aligned} 
\label{eq: derde orde hopf coefficient c}
\end{equation} 
then for $\mu_2 < 0$, the Hopf bifurcation is subcritical; for $\mu_2 > 0$, the Hopf bifurcation is supercritical.
\end{theorem}

In order to apply Theorem \ref{thm: occurence of the Hopf bifurcation} and \ref{thm: direction of the Hopf bifurcation}  to system \eqref{eq: normal form}, we first note that system \eqref{eq: normal form} is equivalent to the following system on $\mathbb{R}^2$:
\begin{equation} 
\begin{aligned}
\begin{pmatrix}
\dot{x}_1(t) \\ \dot{x}_2(t) 
\end{pmatrix}
&=  \begin{pmatrix}
\lambda - K \cos \beta & -1 + K \sin \beta  \\
1 - K \sin \beta & \lambda  - K \cos \beta
\end{pmatrix} \begin{pmatrix}
x_1(t) \\ x_2 (t)
\end{pmatrix}
\\ &\qquad+  \begin{pmatrix}
x_1 (t), & x_2 (t)
\end{pmatrix}
\begin{pmatrix}
x_1(t) \\ x_2 (t)
\end{pmatrix}
\begin{pmatrix}
1 & - \gamma \\
\gamma & 1
\end{pmatrix}
\begin{pmatrix}
x_1 (t) \\ x_2 (t)
\end{pmatrix} \\
&\qquad +  K \begin{pmatrix}
 \cos \beta & -  \sin \beta \\
 \sin \beta & \cos \beta
\end{pmatrix}  
\begin{pmatrix}
x_1 (t - \tau) \\
x_2 (t - \tau) 
\end{pmatrix}
\end{aligned} \label{eq: normal form op R^2}
\end{equation} 
The characteristic matrix of the linearization around zero is given by
\begin{align}
\Delta(\mu, \lambda, \tau) = \mu I - \begin{pmatrix}
\lambda - K \cos \beta & -1 + K \sin \beta  \\
1 - K \sin \beta & \lambda  - K \cos \beta
\end{pmatrix}  -  K e^{- \mu \tau} \begin{pmatrix}
 \cos \beta & -  \sin \beta \\
 \sin \beta & \cos \beta
\end{pmatrix}.   \label{eq: characteristic equation criticality}
\end{align} 
The non--linear term in \eqref{eq: normal form op R^2}, can be given by the function $g: \mathcal{C}([-\tau, 0], \mathbb{R}^2) \times \mathbb{R} \to \mathbb{R}^2$ given by
\begin{align} \label{eq: non-linear term}
g(x_t, \lambda) = \inner{x_t(0), x_t(0)} C x_t(0) \quad \mbox{with} \quad C = \begin{pmatrix}
1 & - \gamma \\
\gamma & 1
\end{pmatrix}.
\end{align}

An application of Theorem \ref{thm: occurence of the Hopf bifurcation} yields the following result.
 
\begin{theorem} \label{lem: conditions hopf over pyragas}
Consider the system \eqref{eq: normal form}. Assume
\begin{align}
1 + 2 \pi K e^{i \beta} &\neq 0 \label{eq: enkelvoudige multipliciteit pyragas curve} 
\end{align}
If 
\begin{align}
1 + 2 \pi K \left[ \cos \beta + \gamma \sin \beta \right] > 0  \label{eq: noemer ongelijk nul pyragas curve 1}
\end{align}
then we find a Hopf bifurcation at $(\lambda, \tau) = (0, 2\pi)$ if we approach the point $(\lambda, \tau) = (0, 2\pi)$ over the extended Pyrags curve from the left.

If 
\begin{align}
1 + 2 \pi K \left[ \cos \beta + \gamma \sin \beta \right] <  0  \label{eq: noemer ongelijk nul pyragas curve 2}
\end{align}
then we find a Hopf bifurcation at $(\lambda, \tau) = (0, 2\pi)$ if we approach the point $(\lambda, \tau) = (0, 2\pi)$ over the extended Pyragas curve from the right.     
\end{theorem}
\begin{proof} 
We note that for $(\lambda, \tau) = (0, 2 \pi)$, $\mu = i$ is a root of the characteristic equation $\det \Delta(z) = $, where $\Delta(z)$ is given by \eqref{eq: characteristic equation criticality}. Using this fact in combination with the definition of $p, q$ as in Theorem \ref{thm: occurence of the Hopf bifurcation}, we find that 
\begin{align}
p = \begin{pmatrix}
1 \\ -i
\end{pmatrix} \label{eq: vector p}, \quad q = \alpha \begin{pmatrix}
1 \\  i
\end{pmatrix}
\end{align}
The normalization factor $\alpha \in \mathbb{C}$ in \eqref{eq: vector p} should be chosen such that
\begin{align}
q \cdot D_1 \Delta(i \omega, \lambda) p = 1 \label{eq: normalisatie alpha}
\end{align}
(see \eqref{eq: definitie p, q, phi}). Using \eqref{eq: characteristic equation criticality}, we note that
\begin{align*}
D_1 \Delta(i, 0, 2 \pi) = I + K \tau e^{- i 2\pi} \begin{pmatrix}
\cos \beta & - \sin \beta \\
\sin \beta & \cos \beta
\end{pmatrix} = I + K \tau \begin{pmatrix}
\cos \beta & - \sin \beta \\
\sin \beta & \cos \beta
\end{pmatrix}
\end{align*}
Thus we find that
\begin{align*}
q \cdot D_1 \Delta(i \omega, \lambda)p &= \alpha \left((1, i) \begin{pmatrix}
1 \\ -i
\end{pmatrix} + K \tau  (1, i) \begin{pmatrix}
\cos \beta & - \sin \beta \\
\sin \beta & \cos \beta
\end{pmatrix} \begin{pmatrix}
1 \\ -i
\end{pmatrix} \right) \\
& = 2 \alpha \left(1 + K \tau e^{i \beta} \right)
\end{align*}
Condition \eqref{eq: normalisatie alpha} therefore yields
\begin{align}
\alpha = \frac{1}{2 (1 + K \tau e^{i \beta})}. \label{eq: waarde alpha}
\end{align}

If we approach the point $(\lambda, \tau) = (0, 2 \pi)$ over the extended Pyragas curve from the left, we can parametrize the path by
\begin{align}
(\lambda(\theta), \tau(\theta)) = \left(\theta, \frac{2 \pi}{1- \gamma \theta} \right), \quad \theta \in \mathbb{R} \backslash \left \{ \frac{1}{\gamma} \right \}. \label{eq: parametrisatie extended pyragas curve links}
\end{align}
Using \eqref{eq: characteristic equation criticality}, we find that, for parameter values on this curve, the characteristic matrix is given by
\begin{align*}
\Delta(\mu, \theta) = \mu I - \begin{pmatrix}
\theta - K \cos \beta & -1 + K \sin \beta \\
1 - K \sin \beta & \theta - K \cos \beta
\end{pmatrix} - K e^{- \mu \tau(\theta)} \begin{pmatrix}
\cos \beta & - \sin \beta \\
\sin \beta & \cos \beta
\end{pmatrix}.
\end{align*}
We are interested in the Hopf bifurcation at $(\lambda, \tau) = (0, 2 \pi)$. We note that the path parametrized by \eqref{eq: parametrisatie extended pyragas curve links} reaches this point for $\theta = 0$. We find that
\begin{align*}
D_2 \Delta(i, 0) &= - \left. \frac{d}{d \theta} \right|_{\theta = 0} \begin{pmatrix}
\theta - K \cos \beta & -1 + K \sin \beta \\
1 - K \sin \beta & \theta - K \cos \beta
\end{pmatrix}\\
&\qquad - K e^{- i \tau(0)} \begin{pmatrix}
\cos \beta & - \sin \beta \\
\sin \beta & \cos \beta
\end{pmatrix} \left(- i \left. \frac{d \tau}{d \theta} \right|_{\theta = 0} \right) \\
& = - I + 2 \pi i K \gamma \begin{pmatrix}
\cos \beta & - \sin \beta \\
\sin \beta & \cos \beta
\end{pmatrix}
\end{align*} 
We note that
\begin{align*}
q  \begin{pmatrix}
\cos \beta & - \sin \beta \\
\sin \beta & \cos \beta
\end{pmatrix} p &= \alpha (1, i)  \begin{pmatrix}
\cos \beta & - \sin \beta \\
\sin \beta & \cos \beta
\end{pmatrix} \begin{pmatrix}
1 \\ -i
\end{pmatrix}\\
&= \alpha (1, i) \begin{pmatrix}
\cos \beta + i \sin \beta \\
\sin \beta - i \cos \beta
\end{pmatrix} \\ &
 = 2 \alpha (\cos \beta + i \sin \beta) = 2 \alpha e^{i \beta}
\end{align*}
Since $\alpha$ is given by \eqref{eq: waarde alpha}, we find that
\begin{align*}
q \cdot D_2 \Delta (i, 0) p = - 2 \alpha + 4 \pi i K \gamma \alpha e^{i \beta} = \frac{-1 + 2 \pi i \gamma K e^{i \beta}}{1 + K \tau e^{i \beta}}
\end{align*}
which gives
\begin{align*}
\re (q \cdot D_2 \Delta (i, 0) p) = - \frac{1 + 2 \pi K (\cos \beta + \gamma \sin \beta)}{\left| 1 + K 2 \pi e^{i \beta} \right|^2}
\end{align*}
We conclude that if \eqref{eq: noemer ongelijk nul pyragas curve 1} holds, we have that $\re (q \cdot D_2 \Delta (i, 0) p) < 0$. Condition \eqref{eq: enkelvoudige multipliciteit pyragas curve} ensures that $\mu = i$ has multiplicity one as a root of $\Delta(\mu, 0)$ and one easily verifies  that $\mu = i$ is the only root of $\Delta(\mu, 0)$ of the form $i \mathbb{Z}$.  Therefore if \eqref{eq: enkelvoudige multipliciteit pyragas curve} --  \eqref{eq: noemer ongelijk nul pyragas curve 1} hold, we obtain a Hopf bifurcation if we approach the point $(\lambda, \tau) = (0, 2 \pi)$ over the extended Pyragas curve from left.

Similarly, if we approach the point $(\lambda, \tau) = (0, 2 \pi)$ over the extended Pyragas curve from the right, we parametrize the path by \eqref{eq: parametrisatie extended pyragas curve links} by replacing $\theta \mapsto - \theta$. Denote by $\tilde{\Delta}$ the characteristic matrix of system \eqref{eq: normal form} for parameter values $(\lambda, \tau)$ on this path. A similar analysis then shows that 
\begin{align*}
\re (q \cdot D \tilde{\Delta}_2 (i, 0) p) =  \frac{1 + 2 \pi K (\cos \beta + \gamma \sin \beta)}{\left| 1 + K 2 \pi e^{i \beta} \right|^2}
\end{align*}
Thus, $\re (q \cdot D \tilde{\Delta}_2 (i, 0) p) < 0$ if \eqref{eq: noemer ongelijk nul pyragas curve 2} is satisfied. Therefore, if \eqref{eq: noemer ongelijk nul pyragas curve 2} and \eqref{eq: enkelvoudige multipliciteit pyragas curve} hold, we find a Hopf bifurcation at $(\lambda, \tau) = (0, 2 \pi)$ if we approach this point over the extended Pyragas curve from the right.
\end{proof}

Now that we have derived conditions for a Hopf bifurcation in the origin to occur, we determine the direction of the bifurcation using Theorem \ref{thm: direction of the Hopf bifurcation}. As outlined before, the direction of the Hopf bifurcation will give us conditions for \po to be (un)stable as a solution of \nf. 

\begin{theorem} \label{lem: curve is extended Pyragas curve from the left}
If we approach the Hopf bifurcation point $(\lambda, \tau) = (0, 2 \pi)$ over the extended Pyragas curve from the left, the value of $\mu_2$ as defined in Theorem \ref{thm: direction of the Hopf bifurcation} is given by
\begin{align*}
\mu_2 = -4
\end{align*} 
If we approach the Hopf bifurcation point $(\lambda, \tau) = (0, 2 \pi)$ over the extended Pyragas curve from the right, the value of $\mu_2$ as defined in Theorem \ref{thm: direction of the Hopf bifurcation} is given by
\begin{align*}
\mu_2 = 4
\end{align*} 
\end{theorem} 

\begin{proof}
Computing the derivative of \eqref{eq: non-linear term} gives (see \cite{scriptie} for more details):
\begin{align} \label{eq: waarden afgeleiden 1}
D_1^2 g(0, \lambda) &= 0 \quad \mbox{for all} \ \lambda \in \mathbb{R} \\
D_1^3g(\phi, \lambda) (f_1, f_2, f_3) &=  \sum_{\sigma \in S_3} \inner{f_{\sigma(1)}(0), f_{\sigma(2)} (0)} C f_{\sigma(3)} (0) \label{eq: waarden afgeleiden 2}
\end{align}
for all $\phi, f_1, f_2, f_3 \in \mathcal{C}\left([-\tau, 0], \mathbb{R}^2 \right)$. Here, $S_3$ denotes the permutation group of three objects. Using this, we find that
\begin{align*}
c &= \frac{1}{2}  q \cdot D_1^3 g(0, \lambda)(\phi, \phi, \overline{\phi}) + 0 + 0\\ &= \frac{1}{2} q \cdot \left(2 \inner{\phi(0), \phi(0)} C \overline{\phi(0)} + 2 \inner{\overline{\phi(0)}, \phi(0)} C \phi(0) + 2 \inner{\phi(0), \overline{\phi(0)}} C \phi(0) \right) \\
&= q \cdot \left(\inner{p, p} C \overline{p} + \inner{p, \overline{p}} C p + \inner{\overline{p}, p} C p \right)  \\&= \frac{4(1 + i \gamma)}{1 + K \tau e^{i \beta}}.
\end{align*}
Taking real parts yields
\begin{align*}
\re c = \frac{4(1 + K \tau \left( \cos (\beta - \phi) + \gamma \sin \beta \right)}{\left| 1 + K \tau e^{i \beta} \right|^2}. 
\end{align*}
Let us now approach the point $(\lambda, \tau) = (0, 2\pi)$ over the extended Pyragas curve from the left. We find as in the proof of Lemma \ref{lem: conditions hopf over pyragas} that 
\begin{align*}
\re (q \cdot \Delta_2 (i, 0) p) = - \frac{1 + 2 \pi K (\cos \beta + \gamma \sin \beta)}{\left| 1 + K 2 \pi e^{i \beta} \right|^2}.
\end{align*}
It follows that $\mu_2 = -4$. 

Similarly, if we approach the point $(\lambda, \tau) = (0, 2 \pi)$ over the extended Pyragas curve from the right, we find as in the proof of Lemma \ref{lem: conditions hopf over pyragas} that 
\begin{align*}
\re (q \cdot \Delta_2 (i, 0) p) =  \frac{1 + 2 \pi K (\cos \beta + \gamma \sin \beta)}{\left| 1 + K 2 \pi e^{i \beta} \right|^2}.
\end{align*}
Combining this with the value of $\re c$, we find that $\mu_2 = 4$. 
\end{proof}

We are now able to determine for which parameter values \eqref{eq: periodic orbit} is (un)stable as a solution of \eqref{eq: normal form}. 

\begin{cor} \label{cor: stabiel of instabiele oplossing}
Let $1 + 2 \pi K e^{i \beta} \neq 0$. If 
\begin{align}
 1 + 2 \pi K \left[ \cos \beta + \gamma \sin \beta\right] > 0 \label{eq: rechts instabiel fixed point}
\end{align}
then for small $\lambda$, \eqref{eq: periodic orbit} is an unstable periodic solution of \eqref{eq: normal form}.  Furthermore, if for $\lambda = 0, \tau = 2 \pi$ no roots of the characterstic equation $\det \Delta (\mu) = 0$ with $\Delta(\mu) $ as in \eqref{eq: characteristic equation criticality} are in the right half of the complex plane and 
\begin{align}
 1 + 2 \pi K \left[ \cos \beta + \gamma \sin \beta\right] < 0 \label{eq: links instabiel fixed point}
\end{align}
then for small $\lambda$, \eqref{eq: periodic orbit} is a stable periodic solution of \eqref{eq: normal form}.   
\end{cor}

\begin{proof}
If \eqref{eq: rechts instabiel fixed point} is satisfied, then Lemma \ref{lem: conditions hopf over pyragas} shows that we find a Hopf bifurcation at the point $(\lambda, \tau) = (0, 2 \pi)$ if we approach this point over the extended Pyragas curve from the left. Combining Lemma \ref{lem: curve is extended Pyragas curve from the left} with Theorem \ref{thm: direction of the Hopf bifurcation}, we find that this Hopf bifurcation is subcritical. Thus, there exists an unstable periodic solution for parameter values $(\lambda, \tau)$ on the (extended) Pyragas curve to the left of the point $(0, 2 \pi)$. By the Hopf bifurcation theorem, the periodic solution for these parameter values is unique. By definition of the Pyragas curve, \eqref{eq: periodic orbit} is a periodic solution of \eqref{eq: normal form} for $(\lambda, \tau)$ near $(0, 2 \pi)$, i.e. this is the periodic solution generated by the Hopf bifurcation. We conclude that for $(\lambda, \tau)$ on the Pyragas curve near $(0, 2 \pi)$, \eqref{eq: periodic orbit} is an unstable periodic solution of \eqref{eq: normal form}. 

If \eqref{eq: links instabiel fixed point} is satisfied, we have by Lemma \ref{lem: conditions hopf over pyragas} that we find a Hopf bifurcation at the point $(\lambda, \tau) = (0, 2 \pi)$ if we approach this point over the extended Pyragas curve from the right. Combining Lemma \ref{lem: curve is extended Pyragas curve from the left} with Theorem \ref{thm: direction of the Hopf bifurcation}, we find that this Hopf bifurcation is supercritical. 

Therefore, we find an unique, stable periodic solution of \eqref{eq: normal form} for $(\lambda, \tau)$ on the Pyragas curve near $(0, 2 \pi)$. Since \eqref{eq: periodic orbit} is a periodic solution of \eqref{eq: normal form} for $(\lambda, \tau)$ on the Pyragas curve, we conclude that for $(\lambda, \tau)$ on the Pyragas curve near $(0, 2 \pi)$, this solution is in fact stable if for $\lambda = 0, \tau = 2 \pi$ no roots of the characterstic equation are in the right half of the complex plane. 
\end{proof}

Recall that in Section \ref{sec: uncontrolled system} we determined the direction of Hopf bifurcation when we vary $\lambda$. A similar approach can be followed for the controlled system \nf to give an alternative proof of Corollary \ref{cor: stabiel of instabiele oplossing} using Lemma \ref{lem: direction hopf bifurcation by roots of characteristic equation}.

\begin{proof} \emph{(of Corollary \ref{cor: stabiel of instabiele oplossing})}
The characteristic function corresponding to the linearization of \nf around $z= 0$ is given by
\begin{align} \label{eq: ce complexwaardig}
\Delta(\mu) = \mu - (\lambda + i) + Ke^{i \beta} \left[1 - e^{- \mu \tau}\right]
\end{align}
We recall from the proof of Lemma \ref{lem: conditions hopf over pyragas} that for $\lambda= 0$, $\mu = i$ is a root of \eqref{eq: ce complexwaardig} and that there are no other roots on the imaginary axis. Furthermore, if $1 + 2 \pi K e^{i \beta} \neq 0$, then $\mu = i$ has multiplicity one as a solution of $\Delta(\mu) = 0$. Therefore, if $\mu = i$ crosses the imaginary axis with non--zero speed as we cross the point $(\lambda, \tau) = (0, 2 \pi)$ over the Pyragas curve, a Hopf bifurcation of the origin occurs for $\lambda = 0$. 

Parametrize the Pyragas curve as in \eqref{eq: parametrisatie extended pyragas curve links} and, for small $\theta$, $\mu = \mu(\theta)$ for the root satisfying $\Delta(\mu(\theta))=0$ for $\lambda = \lambda(\theta)$, and $\tau = \tau(\theta)$ as in \eqref{eq: parametrisatie extended pyragas curve links} with $\mu(0)= i$. Differentiation of \eqref{eq: ce complexwaardig} gives that
\begin{align*}
0 = \left. \frac{d \mu}{d \theta}\right|_{\theta = 0} - 1 + K e^{i \beta} \left( \left. \frac{d \mu}{d \theta}\right|_{\theta = 0} 2 \pi + 2 \pi \gamma i\right) 
\end{align*}
which we can rewrite as
\begin{align*}
 \left. \frac{d \mu}{d \theta}\right|_{\theta = 0} \left(1 + 2 \pi Ke^{i \beta}\right)= 1 - 2 \pi \gamma i Ke^{i \beta}
\end{align*}
which gives
\begin{align*}
 \left. \frac{d \mu}{d \theta}\right|_{\theta = 0} &= \frac{1}{\left|1 + 2 \pi Ke^{i \beta}\right|^2 } \left(1 - 2 \pi \gamma i Ke^{i \beta} \right) \left(1 + 2 \pi K e^{-i \beta} \right) \\
 & =  \frac{1}{\left|1 + 2 \pi Ke^{i \beta}\right|^2 } \left(1 + 2 \pi K e^{- i \beta} - 2 \pi \gamma i Ke^{i \beta} - 4 \pi^2 \gamma K^2 i \right)
\end{align*}
Taking real parts yields
\begin{align*}
 \left. \frac{d \re \mu}{d \theta}\right|_{\theta = 0} = \re  \left. \frac{d \mu}{d \theta}\right|_{\theta = 0} = \frac{1}{\left|1 + 2 \pi Ke^{i \beta}\right|^2 } \left(1 + 2 \pi K \cos \beta + 2 \pi \gamma K \sin \beta \right)
\end{align*}
In particular, if $1 + 2 \pi K (\cos \beta + \gamma \sin \beta) \neq 0$, then the root $\mu = i$ that exists for $(\lambda, \tau) = (0, 2\pi)$ crosses the imaginary axis with non--zero speed as we cross the point $(\lambda, \tau) = (0, 2 \pi)$ over the Pyragas curve. This shows that there is a Hopf bifurcation at the origin. An application of Lemma \ref{lem: direction hopf bifurcation by roots of characteristic equation} now yields the result. 
\end{proof} 

We remark that this alternative proof of Corollary \ref{cor: stabiel of instabiele oplossing} exploits the fact that the extended Pyragas curve is defined in such a way that we a priori know for which points on the curve a periodic solution of the system \nf exists. We will us this observation again in Section \ref{sec: afgeleide control} when we introduce a variation of Pyragas control scheme to system \eqref{eq: uncontrolled normal form}.

\section{Hopf bifurcation and dynamics of the controlled system} \label{sec: lambda varieren}
In the previous section, we approached the Hopf bifurcation point $(\lambda, \tau) = (0, 2\pi)$ over the extended Pyragas curve. As remarked before, there are of course many different ways to approach this bifurcation point. In this section, we approach the bifurcation point parallel to the $\lambda$-axis, as was done in \cite{oddnumberlimitation}. This again enables us to determine stability conditions for \eqref{eq: periodic orbit} as a solution of \eqref{eq: normal form} and gives us more insight in the dynamics of the controlled system.  

Using Theorem \ref{thm: occurence of the Hopf bifurcation}, we can determine conditions for a Hopf bifurcation of system \eqref{eq: normal form} to occur if we vary $\lambda$ and leave all the other parameters fixed. We state the following Lemma without proof: 

\begin{lemma} \label{lem: hopf bifurcation curve}
Let us consider the system \eqref{eq: normal form} where we leave all parameters but $\lambda$ fixed. Let $(\lambda, \tau) \neq (0, 0)$ be such that
\begin{align}
\lambda &= K \left[ \cos \beta - \cos (\beta - \phi) \right] \label{eq: lambda hopf} \\ 
\tau &= \frac{\phi}{1 - K \left[\sin \beta - \sin(\beta - \phi) \right]} \label{eq: tau hopf}
\end{align}
for some $\phi \in \mathbb{R}\backslash \{0\}$. Furthermore, assume that
\begin{align}
1 + K \tau e^{i ( \beta - \phi)} &\neq 0 \label{eq: voorwaarden hopf enkele wortel}  \\ 
1 + K \tau \cos (\beta - \phi) &> 0 \label{eq: voorwaarden hopf afgeleide ongelijk nul}
\end{align} 
Then a Hopf bifurcation of the origin of system \eqref{eq: normal form} occurs.  
\end{lemma}

As in \cite{oddnumberlimitation}, we define the \emph{Hopf bifurcation curve} as the curve in $(\lambda, \tau)$-parameter space parametrized by \eqref{eq: lambda hopf}--\eqref{eq: tau hopf} for $\phi \in \mathbb{R}$. We note that the Pyragas curve (see Definition  \ref{def: pyragas curve}) ends on the Hopf bifurcation point at $(\lambda, \tau) = (0, 2 \pi)$. We can now try to choose the parameters in such a way that the periodic solution \eqref{eq: periodic orbit} of \eqref{eq: normal form} emmanates from a supercritical Hopf bifurcation; then \eqref{eq: periodic orbit} is a stable solution of \eqref{eq: normal form} for parameter values near the bifurcation point.

In \cite{oddnumberlimitation}, the direction of the Hopf bifurcation was determined using a normal form reduction. Here, we rederive this result directly as an application of Theorem \ref{thm: direction of the Hopf bifurcation}. 

\begin{figure}
\centering
\includegraphics[width = 0.5\textwidth]{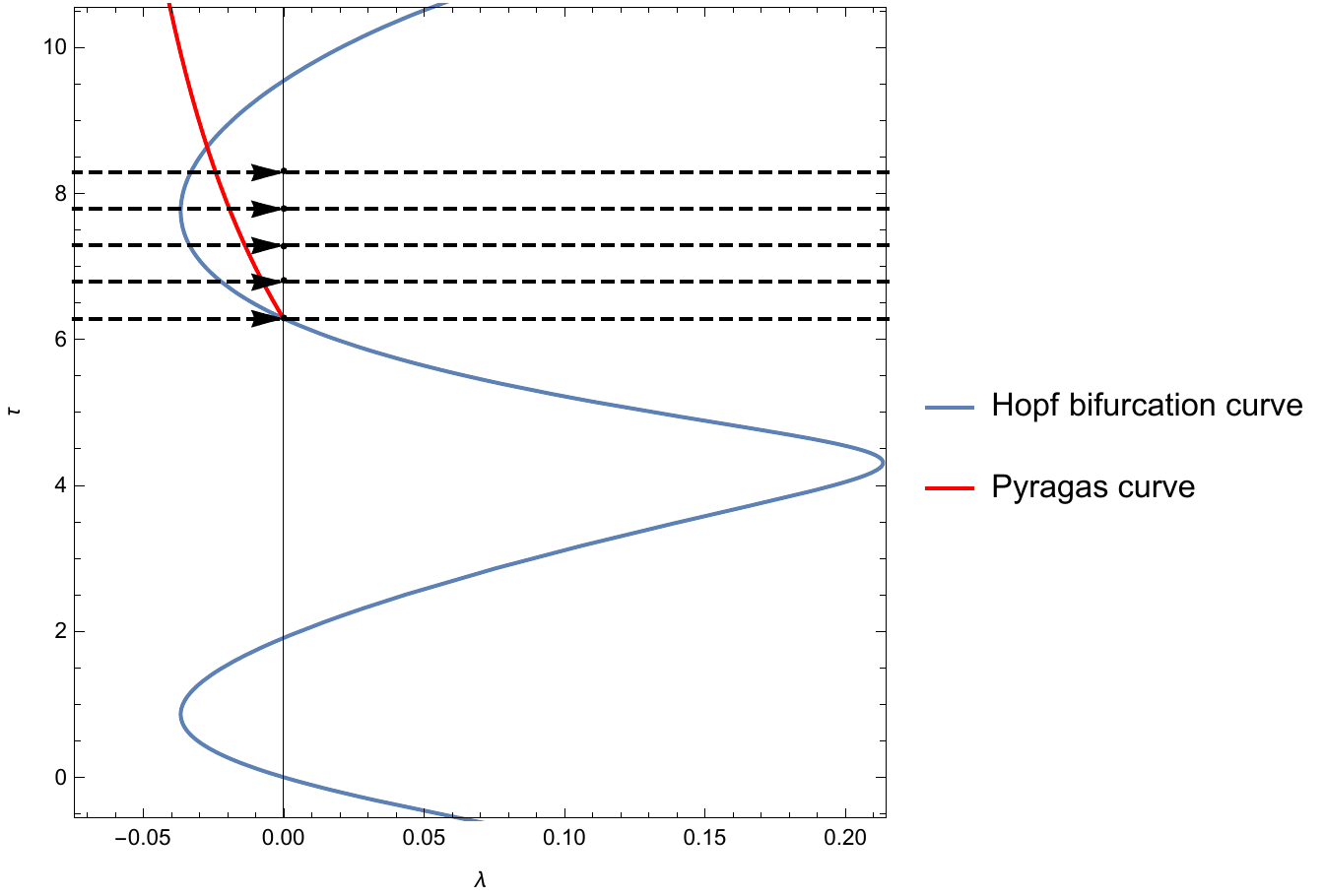}
\caption{Approaching the Hopf bifurcation points parallel to the $\lambda$-axis.}
\end{figure}

\begin{theorem} \label{lem: coefficient parallel aan lambda as}
Let $(\lambda, \tau)$ be a point on the Hopf bifurcation curve and let $\phi \in \mathbb{R}\backslash \{0\}$ satisfy \eqref{eq: lambda hopf}-\eqref{eq: tau hopf}. If $\lambda$ varies while all other parameters remain fixed, then the value of $\mu_2$ as defined in \eqref{eq: derde orde hopf coefficient} is given by 
\begin{align}
\mu_2 = -\frac{4(1 + K \tau \left( \cos (\beta - \phi) + \gamma \sin (\beta - \phi) \right)}{1 + K \tau \cos (\beta - \phi)} \label{eq: mu2 parallel aan lambda as}
\end{align}
\end{theorem}
\begin{proof}
We first calculate $p, q$ as defined in \eqref{eq: definitie p, q, phi}. Set 
\begin{align}
A = \begin{pmatrix}
\lambda - K \cos \beta & -1 + K \sin \beta  \\
1 - K \sin \beta & \lambda  - K \cos \beta
\end{pmatrix}, \quad B = K \begin{pmatrix}
 \cos \beta & -  \sin \beta \\
 \sin \beta & \cos \beta
\end{pmatrix}  \label{eq: matrices characteristic equation}
\end{align}
We recall that if $(\lambda, \tau)$ lies on the Hopf bifurcation curve, then there exsists an $\omega \in \mathbb{R}$ satisfying $\phi = \omega \tau$ such that $\Delta(i \omega, \lambda) = 0$. By definition, $p \in \mathbb{C}^2$ satisfies $\Delta(i \omega, \lambda, \tau)p = 0$ (see \eqref{eq: definitie p, q, phi}). A similar computation as in the proof of Lemma \ref{lem: conditions hopf over pyragas} yields 
\begin{align*}
p = \begin{pmatrix}
1 \\ - i 
\end{pmatrix}, \quad
q  = \frac{1}{2 (1 + K \tau e^{i(\beta- \phi)})} \begin{pmatrix}
1 \\ i
\end{pmatrix}.
\end{align*}
Using \eqref{eq: characteristic equation criticality}, we obtain
\begin{align*}
D_2 \Delta(i \omega_0, \lambda, \tau) = - I
\end{align*}
which gives
\begin{align*}
q \cdot D_2 \Delta(i \omega_0, \lambda) p &= - q \cdot p  \\
&= - \frac{1}{1 + K \tau e^{i (\beta - \phi)} }.
\end{align*}
Taking the real part yields 
\begin{align*}
\re (q \cdot D_2 \Delta(i \omega_0, \lambda) p) = -\frac{1 + K \tau \cos (\beta - \phi)}{\left|1 + K \tau e^{i (\beta - \phi)} \right|^2}
\end{align*}

Using \eqref{eq: waarden afgeleiden 1} -- \eqref{eq: waarden afgeleiden 2}, we can now explicitly compute $c$:
\begin{align*}
c &= \frac{1}{2}  q \cdot D_1^3 g(0, \lambda)(\phi, \phi, \overline{\phi}) + 0 + 0\\ &= \frac{1}{2} q \cdot \left(2 \inner{\phi(0), \phi(0)} C \overline{\phi(0)} + 2 \inner{\overline{\phi(0)}, \phi(0)} C \phi(0) + 2 \inner{\phi(0), \overline{\phi(0)}} C \phi(0) \right) \\
&= q \cdot \left(\inner{p, p} C \overline{p} + \inner{p, \overline{p}} C p + \inner{\overline{p}, p} C p \right) \\ &= \frac{4(1 + i \gamma)}{1 + K \tau e^{i( \beta - \phi)}}
\end{align*}
Thus we find
\begin{align*}
\re c = \frac{4(1 + K \tau \left( \cos (\beta - \phi) + \gamma \sin (\beta - \phi) \right)}{\left| 1 + K \tau e^{i(\beta - \phi)} \right|^2}.
\end{align*}
Using the definition of $\mu_2$ as in Theorem \ref{thm: direction of the Hopf bifurcation}, we arrive at equation \eqref{eq: mu2 parallel aan lambda as}. This completes the proof. 
\end{proof}

We are also able to determine the direction of the Hopf bifurcation for parameter values $(\lambda, \tau)$ for which a Hopf bifurcation of the origin of system \eqref{eq: normal form} occurs; cf. eq. (8) in \cite{oddnumberlimitation}. 

\begin{cor} \label{cor: directie lambda parameter}
Let $(\lambda, \tau)$ be such that a Hopf bifurcation of the origin of system \eqref{eq: normal form} occurs, i.e., let the conditions of Theorem \ref{lem: hopf bifurcation curve} be satisfied for some $\phi \in \mathbb{R}\backslash \{0\}$. If  
\begin{align}
1 + K \tau \left[ \cos (\beta - \phi) + \gamma \sin(\beta - \phi) \right] > 0 \label{eq: lambda varieren subcritical}
\end{align}
then the Hopf bifurcation at $(\lambda, \tau)$ is subcritical. If 
\begin{align}
1 + K \tau \left[ \cos (\beta - \phi) + \gamma \sin(\beta - \phi) \right] < 0 \label{eq: lambda varieren supercritical}
\end{align}
the Hopf bifurcation at $(\lambda, \tau)$ is supercritical. 
\end{cor}
\begin{proof}
If the conditions of Theorem \ref{lem: hopf bifurcation curve} are satisfied, then \eqref{eq: voorwaarden hopf afgeleide ongelijk nul} holds and
\begin{align*}
1 + K \tau \cos (\beta - \phi) > 0.
\end{align*}
Combining this inequality with Theorem \ref{lem: coefficient parallel aan lambda as}, we find that $\mu_2 < 0$ if \eqref{eq: lambda varieren subcritical} holds. Using Theorem \ref{thm: direction of the Hopf bifurcation} this shows that the Hopf bifurcation is subcritical. Similarly, if \eqref{eq: lambda varieren supercritical} holds, then $\mu_2 > 0$ and again by Theorem \ref{thm: direction of the Hopf bifurcation} the Hopf bifurcation is supercritical. 
\end{proof}

We can determine the orientation of the Pyragas curve with respect to the Hopf bifurcation curve at the point $(\lambda, \tau) = (0, 2 \pi)$ by computing the slopes of the curves at $(\lambda, \tau) = (0, 2 \pi)$. Combining this with the direction of the Hopf bifurcation curve, we are able to give conditions for \eqref{eq: periodic orbit} to be  (un)stable as a solution of \eqref{eq: normal form}. If the Hopf bifurcation at $(\lambda, \tau) = (0, 2\pi)$ is subcritical and the Pyragas curve is locally to the left of the Hopf bifurcation curve, we expect the periodic solution \eqref{eq: periodic orbit}, that exists for parameter values on the Pyragas curve, to arise from the Hopf bifurcation and therefore be unstable. By an analogous argument, we find that the solution \eqref{eq: periodic orbit} of \eqref{eq: normal form} is stable if the Hopf bifurcation at $(\lambda, \tau) = (0, 2\pi)$ is supercritical and the Pyragas curve is locally to the right of the Hopf bifurcation curve. Following \cite{oddnumberlimitation}, this leads to the following Corollary:

\begin{cor} \label{cor: condities stabiliteit lambda varieren}
Let the parameters $K, \beta, \gamma$ be such that a Hopf bifurcation of system \eqref{eq: normal form} occurs for $(\lambda, \tau) = (0, 2 \pi)$, i.e. let 
\begin{align}
1 + 2 \pi K e^{i \beta} &\neq 0 \label{eq: condities hopf bifurcatie 1} \\ 
1 + 2 \pi K \cos \beta &> 0 \label{eq: condities hopf bifurcatie 2}
\end{align}
If $1 + 2 \pi K \left[ \cos \beta + \gamma \sin \beta \right]< 0$ and the Pyragas curve is locally to the right of the Hopf bifurcation curve, then the periodic solution \eqref{eq: periodic orbit} of \eqref{eq: normal form} is stable for small $\lambda$. If $1 + 2 \pi K \left[ \cos \beta + \gamma \sin \beta \right] >  0$ and the Pyragas curve is locally to the left of the Hopf bifurcation curve, then the periodic solution \eqref{eq: periodic orbit} of \eqref{eq: normal form} is unstable for small $\lambda$. 
\end{cor}

As we have seen in Sections \ref{sec: pyragas curve} -- \ref{sec: lambda varieren}, applying the Hopf bifurcation theorem with respect to different curves yields different results. Comparing Corollary \ref{cor: condities stabiliteit lambda varieren} with Corollary \ref{cor: stabiel of instabiele oplossing}, we see that Corollary \ref{cor: stabiel of instabiele oplossing} gives us weaker conditions for \eqref{eq: periodic orbit} to be (un)stable as a solution of \eqref{eq: normal form} for small $\lambda$. In particular, we can drop the condition \eqref{eq: condities hopf bifurcatie 2} and we no longer have to take the orientation of the Pyragas curve with respect to the Hopf bifurcation curve into account. Using Corollary \ref{cor: stabiel of instabiele oplossing}, we are therefore able to determine upon the (in)stability of the periodic solution \eqref{eq: periodic orbit} of \eqref{eq: normal form} for a wider range of parameter values than if we use Corollary \ref{cor: directie lambda parameter}. 

The approach we have used in Section \ref{sec: lambda varieren} gives more insight in the dynamics of the controlled system \eqref{eq: normal form}. If $1 + K \tau \left[ \cos \beta + \gamma \sin \beta \right] > 0$, then \eqref{eq: lambda varieren subcritical} holds for $\phi$ in a small neighbourhood of $2 \pi$. Applying Corollary \ref{cor: directie lambda parameter}, we find that for parameter values $(\lambda, \tau)$ in a neighbourhood of $(\lambda, \tau) =(0, 2\pi)$ to the left of the Hopf bifurcation curve, a periodic orbit exists. Similarly, if $1 + K \tau \left[ \cos \beta + \gamma \sin \beta \right] < 0$, a periodic orbit exists for all parameter values $(\lambda, \tau)$ in a neighbourhood of $(\lambda, \tau) = (0, 2\pi)$ to the right of the Hopf bifurcation curve. We conclude that by applying Pyragas control, a new set of periodic orbits is created, see also \cite{globalbifurcationanalysis}.

\section{A variation in control term} \label{sec: afgeleide control}
In previous sections, we discussed three different methods to determine the stability of periodic orbit \eqref{eq: periodic orbit} of system \eqref{eq: normal form}. In this section, we return to the general problem of Pyragas control. Let us study the system
\begin{align}
\dot{x}(t) = f(x(t)), \quad x(0) = x_0 \label{eq: uncontrolled system conclusie}
\end{align}
with $f: \mathbb{R}^n \to \mathbb{R}^n$. Let us assume that an unstable periodic solution $u(t)$ of this system exists; denote its period by $T$. In the Pyragas control scheme, we add a term to the system \eqref{eq: uncontrolled system conclusie} in such a way that the periodic solution $u(t)$ is a also a solution of the controlled system. Usually, we write for the controlled system
\begin{align}
\dot{x}(t) = f(x(t)) + K \left[x(t) -x(t-T) \right] \label{eq: pyragas control conclusie}
\end{align} 
There are, however, variations to this scheme possible. We remark that $u(t)$ is also a periodic solution of the system 
\begin{align}
\dot{x}(t) = f(x(t)) + K_1 \left[x(t) - x(t - T) \right] + K_2 \left[\dot{x}(t) -\dot{x}(t - T) \right] \label{eq: pyragas control afgeleide}
\end{align}
We can investigate for which values of $K_1, K_2$ the solution $u(t)$ of \eqref{eq: pyragas control afgeleide} is stable, and how these values of $K_1, K_2$ compare to the values of $K$ for which $u(t)$ is stable as a solution to \eqref{eq: pyragas control conclusie}. 

Applying the type of control given in \eqref{eq: pyragas control afgeleide} yields the system
\begin{equation}
\begin{aligned} \label{eq: dubbele control}
\dot{z}(t) &= (\lambda + i)z(t) + (1+ i \gamma)\left|z(t)\right|^2 z(t) - K_1 e^{i \beta_1} \left[ z(t) - z(t - \tau) \right] \\ &\qquad- K_2 e^{i \beta_2} \left[\dot{z}(t) - \dot{z}(t - \tau) \right]
\end{aligned}
\end{equation}
which we be rewritten as
\begin{equation} \label{eq: nfde goede vorm}
\begin{aligned}
\dot{z}(t) - \frac{K_2 e^{i \beta_2}}{1 + K_2 e^{i \beta_2}} \dot{z}(t - \tau) &= \frac{1}{1 + K_2 e^{i \beta_2}} \left( (\lambda + i)z(t) + (1+ i \gamma)\left|z(t)\right|^2 z(t)\right) \\ &- \frac{K_1 e^{i \beta_1}}{1 + K_2 e^{i \beta_2}} \left[ z(t) - z(t - \tau) \right].
\end{aligned}
\end{equation}
We note that \eqref{eq: nfde goede vorm} is a neutral functional differential equation. Neutral functional differential equations have very different properties from retarded functional differential equations. For example, for retarded functional differential equations the solution operator $T(t)$ is compact for $t \geq r$ (where $r$ denotes the delay of the system), but for neutral functional differential equations this property does in general not hold. Also, if we fix $\alpha, \beta \in \mathbb{R}$, then for neutral functional differential equations we can have an infinite number of roots of the characteristic equation in a strip $\{z \in \mathbb{C} \mid \alpha \leq z \leq \beta \}$. This cannot occur of retarded functional differential equations. Since we can have an infinite number of eigenvalues in a strip $\{z \in \mathbb{C} \mid \alpha \leq z \leq \beta \}$, it can also occur that all the eigenvalues are in the left half of the complex plane, but the eigenvalues get arbritrary close to the imaginary axis. In this case, it is possible that all eigenvalues are in the left half of the complex plane, but the fixed point of the equation is not stable. However, if we have a so--called spectral gap, i.e. there exists a $\gamma < 0$ such that all the eigenvalues are in the set $\{z \in \mathbb{C} \mid \re z < \gamma \}$, then stability of the fixed point is guaranteed. In the case of a spectral gap, we can use the same methods as in the retarded case to find a Hopf bifurcation theorem for neutral equations. 

\begin{lemma} \label{lem: conditions dubbele control hopf over pyragas}
Let $K_1, K_2, \beta_1, \beta_2$ be such that for $\lambda = 0$, there exists a $\gamma < 0$ such that all roots, expect the root $\mu = i$, of \eqref{eq: ce dubbele control} are in the set $\{z \in \mathbb{C} \mid \re z < \gamma \}$. If 
\begin{align*}
1 + 2 \pi K_1 \left( \cos(\beta_1) + \gamma \sin (\beta_1) \right) - 2 \pi K_2 \left(\sin(\beta_2) - \gamma \cos (\beta_2) \right) > 0
\end{align*}
then the periodic solution \po of \eqref{eq: dubbele control} that exists for $\lambda < 0$ is unstable for small $\lambda < 0$. If 
\begin{align} \label{eq: dubbele control stabiliteitsconditie}
1 + 2 \pi K_1 \left( \cos(\beta_1) + \gamma \sin (\beta_1) \right) - 2 \pi K_2 \left(\sin(\beta_2) - \gamma \cos (\beta_2) \right) < 0
\end{align}
the periodic solution \po of \eqref{eq: dubbele control} that exists for $\lambda < 0$ is stable for small $\lambda < 0$.
\end{lemma}
\begin{proof}
We note that the characteristic equation corresponding to the linearization of \eqref{eq: dubbele control} around $z = 0$ is given by
\begin{align} \label{eq: ce dubbele control}
\Delta(\mu) = \mu - (\lambda + i) + K_1 e^{i \beta_1} (1 - e^{- \mu \tau}) + K_2 e^{i \beta_2} \mu (1 - e^{- \mu \tau})
\end{align}
We have that $\Delta(i) = 0$ for $\lambda = 0$ and $\tau = 2 \pi$. We determine whether the root $\mu =i$ moves in our out of the right half of the complex plane if approach the point $(\lambda, \tau) =(0, 2 \pi)$ over the extended Pyragas curve from the left. 

Parametrize the extended Pyragas curve as in \eqref{eq: parametrisatie extended pyragas curve links}. For $\theta$ near 0, write $\mu = \mu(\theta)$ satisfying $\Delta(\mu(\theta)) = 0$ for $\lambda = \lambda(\theta)$ and $\tau = \tau(\theta)$ with $\mu(0)  =i$. Then differentation of \eqref{eq: ce dubbele control} with respect to $\theta$ yields
\begin{align*}
0 &= \left. \frac{d \mu}{d \theta}\right|_{\theta = 0} - 1 + K_1 e^{i \beta_1} e^{- \mu(0) \tau(0)} \left( \mu(\theta) \left. \frac{d \tau}{d \theta} \right|_{\theta = 0} + \left. \frac{d \mu}{d \theta}\right|_{\theta = 0} \tau(0) \right)  \\ 
&+ K_2 e^{i \beta_2} \left. \frac{d \mu}{d \theta}\right|_{\theta = 0} (1 - e^{- \mu(0) \tau(0)}) + K_2 e^{i \beta_2} \mu(0) \left(  \mu(0) \left. \frac{d \tau}{d \theta} \right|_{\theta  = 0} + \left. \frac{d \mu}{d \theta}\right|_{\theta = 0} \tau(0) \right) \\
& = \left. \frac{d \mu}{d \theta}\right|_{\theta = 0} - 1 + K_1e^{i \beta_1} \left(2 \pi i\gamma + 2 \pi \left. \frac{d \mu}{d \theta}\right|_{\theta = 0}\right) + K_2 e^{i \beta_2} i \left(2 \pi i \gamma + 2 \pi \left. \frac{d \mu}{d \theta}\right|_{\theta = 0} \right)
\end{align*}
which can be rewritten as
\begin{align*}
\left. \frac{d \mu}{d \theta}\right|_{\theta = 0} \left(1 + 2 \pi K_1 e^{i \beta_1} + 2 \pi i K_2 e^{i \beta_2} \right) = 1 - 2 \pi \gamma i K_1 e^{i \beta_1} + 2 \pi \gamma K_2 e^{i \beta_2}. 
\end{align*}
With $a = 1 + 2 \pi K_1 e^{i \beta_1} + 2 \pi i K_2 e^{i \beta_2}$ this gives
\begin{align*}
\left. \frac{d \mu}{d \theta}\right|_{\theta = 0} &= \frac{1}{\left| a \right|^2} \left(1 - 2 \pi \gamma i K_1 e^{i \beta_1} + 2 \pi \gamma K_2 e^{i \beta_2}  \right) \left( 1 + 2 \pi K_1 e^{- i \beta_1} - 2 \pi i K_2 e^{- i \beta_2}\right) \\
& = \frac{1}{\left| a \right|^2} \left( 1 + 2 \pi K_1 e^{- i \beta_1} - 2 \pi K_2 i e^{-i \beta_2} - 2 \pi \gamma i K_1 e^{i \beta_1} - 4 \pi^2 \gamma i K_1^2 \right. \\  &- \left.  K_1 K_2 4 \pi^2 \gamma e^{i(\beta_1 - \beta_2)} + 2 \pi \gamma K_2 e^{i \beta_2} + 4 \pi^2 \gamma K_1 K_2 e^{i(\beta_2 - \beta_1)} - 4 \pi^2 \gamma i K_2^2 \right).
\end{align*}
After taking the real part we arrive at
\begin{align*}
\re \left. \frac{d \mu}{d \theta}\right|_{\theta = 0} &= \left. \frac{d \re \mu}{d \theta}\right|_{\theta = 0} \\&= 1 + 2 \pi K_1 \cos \beta_1 - 2 \pi K_2 \sin \beta_2 + 2 \pi \gamma K_1 \sin \beta_1 + 2 \pi \gamma K_2 \cos \beta_2 \\
& = 1 + 2 \pi K_1 \left(\cos \beta_1 + \gamma \sin \beta_1\right) - 2 \pi K_2 (\sin \beta_2 - \gamma \cos \beta_2).
\end{align*}
If $1 + 2 \pi K_1 \left(\cos \beta_1 + \gamma \sin \beta_1\right) - 2 \pi K_2 (\sin \beta_2 - \gamma \cos \beta_2) \neq 0$ and for $\lambda = 0$ all the roots of \eqref{eq: ce dubbele control} except $\mu = i$ are in the left half of the complex plane, then the conditions of the Hopf bifurcation theorem for neutral functional differential equations are satisfied. An application of Lemma \ref{lem: direction hopf bifurcation by roots of characteristic equation} now yields the result.  
\end{proof}

Let us study the case $\gamma = -10, \beta_1 = \beta_2 = \frac{\pi}{4}$. In order to apply Lemma \ref{lem: conditions dubbele control hopf over pyragas} we are interested in values of $K_1, K_2$ such that there exists a $\gamma < 0$ such that all roots, expect the root $\mu = i$, of \eqref{eq: ce dubbele control} are in the set $\{z \in \mathbb{C} \mid \re z < \gamma \}$. We note that if 
\begin{align} \label{eq: stable d operator}
\left| \frac{K_2 e^{i \beta_2}}{1 + K_2 e^{i \beta_2}} \right| < 1
\end{align}
(i.e. we have a stable $D$--operator), then this condition is automatically satisfied. Now let us choose $K_1$ close to zero; using DDEBiftool, we find that for $K_2 = 0$ and some (fixed) $K_1$ small, the characterstic equation \eqref{eq: ce dubbele control} has no roots in the right half of the complex plane. Since for the case $K_2 = 0$, \eqref{eq: dubbele control} reduces to a retarded equation, we automatically have a spectral gap in this case. One can proof that a root of \eqref{eq: ce dubbele control} must cross the imaginary axis to move form the left to the right half of the complex plane. Using this, 
one can draw a stability chart to show that for points inside the region whose boundary is parametrized by
\begin{equation} \label{eq: parametrisatie}
\begin{aligned}
K_1 &= \frac{1}{2 \sin (\omega \pi)} (1 - \omega)\cos(\omega \pi - \beta) \\
K_2 &= \frac{1}{2 \omega \sin (\omega \pi)} (1 - \omega)\sin(\omega \pi - \beta)
\end{aligned}
\end{equation}
with $\omega \in (0, 2)$
no roots of \eqref{eq: ce dubbele control} are in the right half of the complex plane (the region enclosed by the curve in Figure \ref{fig:region}). Thus, if we $K_1, K_2$ are inside the region enclosed by the curve in Figure \ref{fig:region} and the condition \eqref{eq: stable d operator} is satisfied, we have a spectral gap. If then also \eqref{eq: dubbele control stabiliteitsconditie} is satisfied, we can apply  Lemma \ref{lem: conditions dubbele control hopf over pyragas} to find that the periodic solution \po of \eqref{eq: dubbele control} is stable for small $\lambda < 0$ (see Figure \ref{fig:region}). 

We can of course also choose $\beta_1 \neq \beta_2$: see Figure \ref{fig:region2} for the case where we have chosen $\gamma = -10, \beta_1 = - \frac{\pi}{4}$ and $\beta_2 = \frac{3 \pi}{4}$.
\begin{figure}[h!]
\hspace{-20pt}
\begin{subfigure}{0.4\textwidth}
\centering
        \includegraphics[width=\textwidth]{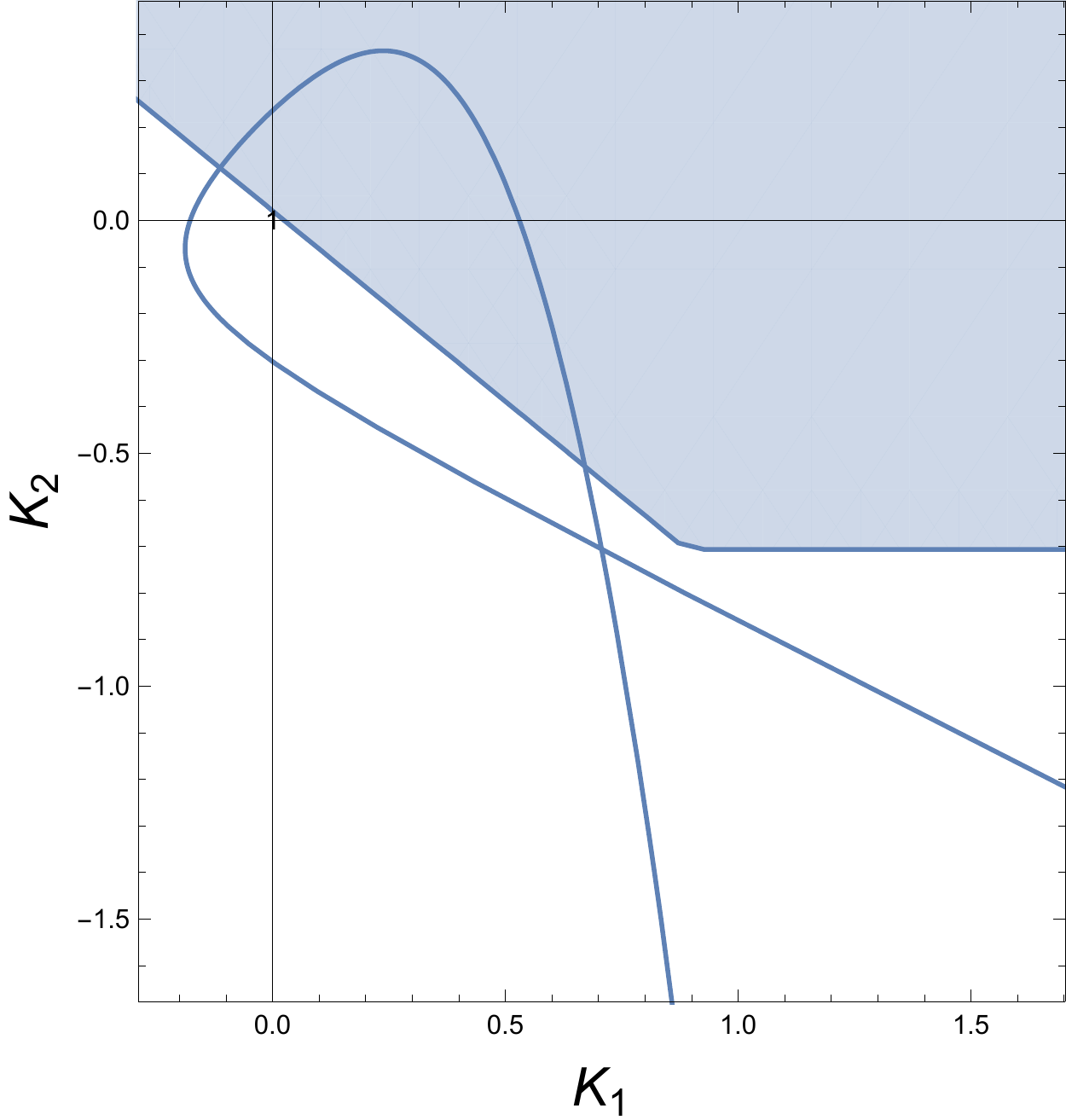}
        \caption{Curve parametrized by \eqref{eq: parametrisatie} for $\omega \in (0, 2)$ with $\beta = \frac{\pi}{4}$, $\gamma = -10$. The shaded region inside indicates the region where both the conditions \eqref{eq: dubbele control stabiliteitsconditie} and \eqref{eq: stable d operator} are satisfied.}
        \label{fig:region}
\end{subfigure}
\hspace{20pt}
\begin{subfigure}{0.4\textwidth}
\centering
        \includegraphics[width = \textwidth]{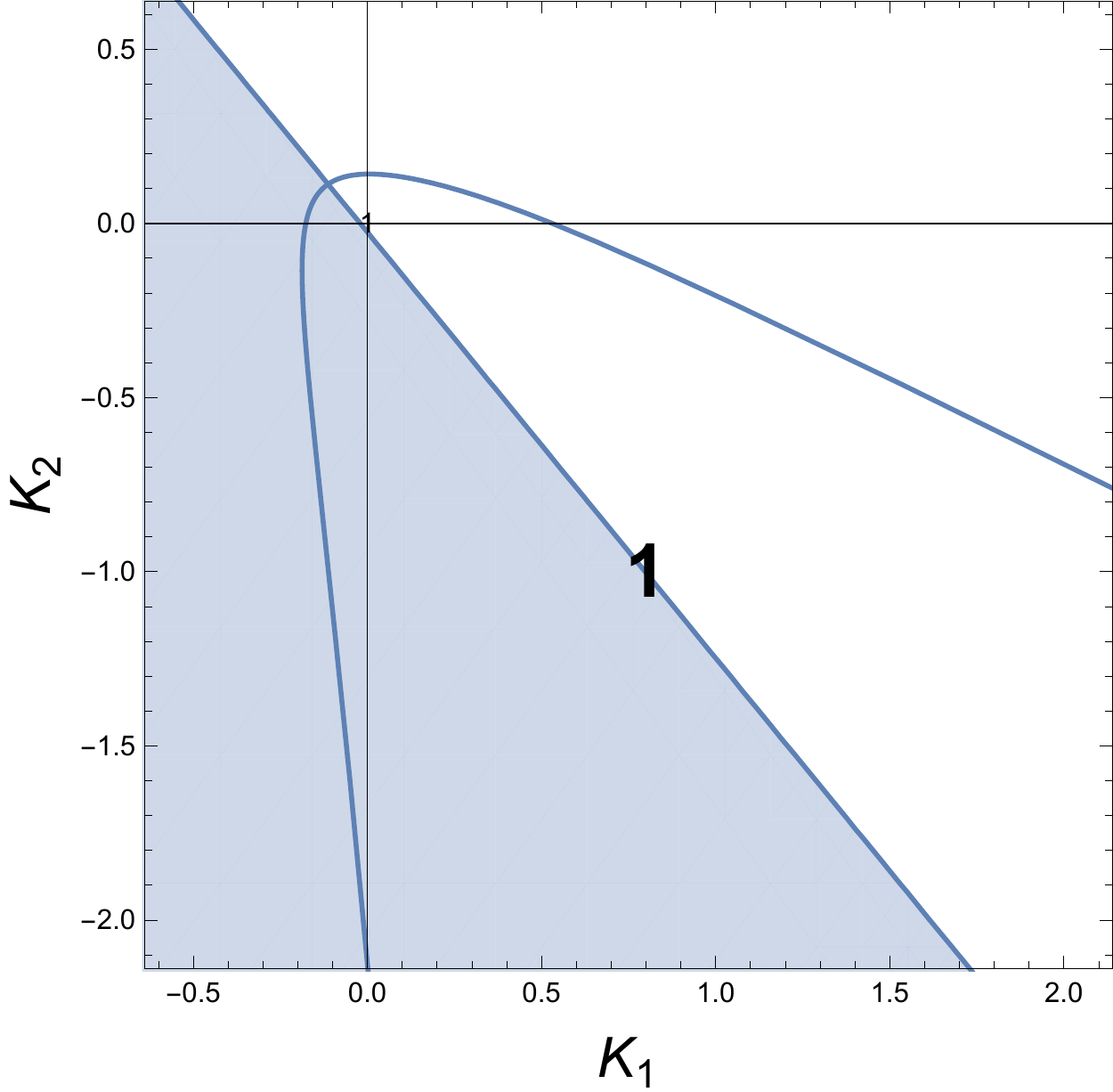}
        \caption{For the case $\gamma = -10, \beta_1 = -\frac{\pi}{4}, \beta_2 = \frac{3}{4} \pi$, we have no roots in the right half of the complex plane for $(K_1, K_2)$ inside the curve (region 1). The shaded region inside indicates the region where both the conditions \eqref{eq: dubbele 	control stabiliteitsconditie} and \eqref{eq: stable d operator} are satisfied.}
        \label{fig:region2}
\end{subfigure}
\end{figure}

Now that we have determined stability conditions for \eqref{eq: periodic orbit} to be stable as a solution of \eqref{eq: dubbele control}, a number of questions arise naturally. For the specific example discussed here, one is interested how the range of values of $\lambda$ for which the periodic orbit \po is (un)stable as a solution of \eqref{eq: dubbele control} compares to the range of values of $\lambda$ for which \po is (un)stable as a solution of \nf. Furthermore, if $\po$ is stable as a solution of both \eqref{eq: dubbele control} and \nf, it is also interesting to study how the basin of attraction in both situations compare. More generally, one would like to apply the control scheme \eqref{eq: pyragas control afgeleide} to various systems or consider different control schemes including a `neutral term'. We hope to return to these questions in the future.

\bibliographystyle{amsplain}

\end{document}